\documentclass[10pt,leqno,twoside]{amsart}
\numberwithin{equation}{section}
\usepackage{latexsym}
\usepackage{amsmath}
\usepackage{amssymb}
\usepackage{mathrsfs}
\usepackage{graphicx}
\usepackage{ifthen}
\usepackage[T1]{fontenc}
\usepackage[latin1]{inputenc}
\usepackage{color}
\newtheorem{Satz}{Theorem}[section]
\newtheorem{lemma}[Satz]{Lemma}
\newtheorem{proposition}[Satz]{Proposition}
\newtheorem{corollary}[Satz]{Corollary}
\newtheorem{theorem}[Satz]{Theorem}
\newtheorem{remark}[Satz]{Remark}

\newtheorem{Def}[Satz]{Definition}

\begin{document}

\newcommand{\C}{\mathbb{C}}
\newcommand{\N}{\mathbb{N}}
\newcommand{\R}{\mathbb{R}}
\newcommand{\Z}{\mathbb{Z}}
\newcommand{\hp}{\mathbb{P}}

\renewcommand{\AA}{\mathcal{A}}
\newcommand{\cB}{\mathcal{B}}
\newcommand{\DD}{\mathcal{D}}
\newcommand{\FF}{\mathcal{F}}
\newcommand{\HH}{\mathcal{H}}
\newcommand{\II}{\mathcal{I}}
\newcommand{\KK}{\mathcal{K}}
\newcommand{\LL}{\mathcal{L}}
\newcommand{\MM}{\mathcal{M}}
\newcommand{\PP}{\mathcal{P}}
\newcommand{\QQ}{\mathcal{Q}}
\newcommand{\RR}{\mathcal{R}}
\newcommand{\TT}{\mathcal{T}}
\newcommand{\cU}{\mathcal{U}}
\newcommand{\cG}{\mathcal{G}}
\newcommand{\hitr}{\HH\TT}
\newcommand{\cA}{{\mathcal A}}
\newcommand{\la}{\lambda}
\newcommand{\ve}{\varepsilon}
\newcommand{\vmo}{VMO($\R^n$)}
\newcommand{\vmom}{\mathrm{VMO}(\R^n)}
\newcommand{\bmo}{BMO($\R^n$)}
\newcommand{\bmom}{\mathrm{BMO}(\R^n)}
\newcommand{\vmoN}{VMO($\R^n; \C^{N \times N}$)}
\newcommand{\vmomN}{\mathrm{VMO}(\R^n; \C^{N \times N})}
\newcommand{\bmoN}{BMO($\R^n; \C^{N \times N}$)}
\newcommand{\bmomN}{\mathrm{BMO}(\R^n; \C^{N \times N})}
\newcommand{\Liloc}{L^1_{\rm loc}}
\newcommand{\Lpo}{L^p_\omega(\R^n)}
\newcommand{\Wmpo}{W^{m,p}_\omega(\R^n)}
\newcommand{\mint}[1]{{\int \!\!\!\!\!\! -}_{\!\!#1}\;}
\newcommand{\cz}{Calder\'on-Zygmund}
\newcommand{\vn}{vspace{.2cm}\noindent}
\newcommand{\dd}{{\rm d}}
\newcommand {\<}{\left\langle}
\renewcommand {\>}{\right\rangle}
\renewcommand {\=}[1]{\stackrel{\text{#1}}{=}}
\renewcommand {\Im}{{\rm Im}}
\renewcommand {\Re}{{\rm Re}}
\newcommand{\HR}{Hil\-bert\-raum}
\newcommand{\Hik}{$H^\infty$-calculus}
\newcommand{\K}{\ensuremath\mathbb{K}}
\newcommand{\vr}{\ensuremath\varrho}
\newcommand{\eps}{\ensuremath\varepsilon}
\newcommand{\id}{\ensuremath\mathrm{Id}}
\newcommand{\lap}{\ensuremath\Delta}
\newcommand{\dv}{\mbox{div}\,}
\newcommand{\grad}{\mbox{grad }}
\newcommand{\rot}{\mbox{rot} }
\newcommand{\curl}{\mbox{curl}\,}
\renewcommand{\L}{\mathscr L}
\newcommand{\supp}{{\mathrm {supp}\ }}
\newcommand{\We}{\mbox{We }}
\newcommand{\Rey}{\mbox{Re }}
\renewcommand{\d}{{\ \mathrm {d}}}
\newcommand{\bog}{Bogovski\u{\i}}
\newcommand{\dist}{\ensuremath\mathrm{dist}}
\newcommand{\ran}{\ensuremath\mathrm{ran}}
\newcommand{\diam}{\ensuremath\mathrm{diam}}
\newcommand{\hr}{{\R^n_+}}
\newcommand{\lpb}{{L^p(B(x_0,r))^N}}
\newcommand{\lpbl}{{L^p(B(x_0,|\lambda|^{-1/m}))^N}}
\newcommand{\lpbs}{{L^p(B(x_0,(s+1)r))^N}}
\newcommand{\lp}{{L^p(\R^n)^N}}
\newcommand{\linfb}{{L^\infty(B(x_0,r))^N}}
\newcommand{\linfbl}{{L^\infty(B(x_0,|\lambda|^{-1/m}))^N}}
\newcommand{\lpbsl}{{L^p(B(x_0,(s+1)|\lambda|^{-1/m}))^N}}
\newcommand{\loc}{\mathrm{loc}}
\newcommand{\nn}{\nonumber}

\newcommand{\ie}{\mathcal I_\varepsilon}
\def\vn{\vspace{.2cm}\noindent}
\def\cF{{\mathcal F}}
\def\mF{{\mathfrak F}}
\renewcommand{\theenumi}{\alph{enumi}}
\renewcommand{\labelenumi}{\theenumi)}

\title[Global existence results for Oldroyd-B Fluids]
{Global existence results for Oldroyd-B Fluids in exterior domains: The case of non-small
coupling parameters}

\author{Daoyuang Fang, Matthias Hieber and Ruizhao Zi}

\address{Department of Mathematics, Zhejiang University,
Hangzhou 310027, China}

\address{Technische Universit\"at Darmstadt, Fachbereich Mathematik,
Schlossgartenstr. 7, D-64289 Darmstadt, Germany and \hfill\break
Center of Smart Interfaces, Petersenstr. 32, D-64287 Darmstadt}

\address{Department of Mathematics, Zhejiang University,
Hangzhou 310027, China}

\email{dyf@zju.edu.cn}

\email{hieber@mathematik.tu-darmstadt.de}

\email{ruizhao3805@163.com}



\begin{abstract}
Consider the set of equations describing Oldroyd-B fluids in an exterior domain. It is shown
that this set of equations admits a unique, global solution in a certain function space provided
the initial data, but not necessarily the coupling constant, is small enough.
\end{abstract}

\maketitle

\vn
\section{Introduction}
\noindent
The theory of Oldroyd-B fluids recently gained
quite some attention. This type of fluids is  described by the following set of equations
\begin{eqnarray}\label{OB}
\begin{cases}
\begin{array}{rrll}
Re(\partial_tu+(u\cdot\nabla) u)-(1-\alpha)\Delta u+\nabla
p&=&\dv\tau&\mathrm{in}\ \ \Omega\times(0,T),\\
\dv u&=&0&\mathrm{in}\ \ \Omega\times(0,T),\\
We(\partial_t\tau+(u\cdot\nabla)\tau+g_a(\tau,
\nabla u))+\tau&=&2\alpha D(u)&\mathrm{in}\ \ \Omega\times(0,T),\\
u&=&0&\mathrm{on}\ \ \partial\Omega\times(0,T),\\
u(0)&=&u_0&\mathrm{in}\ \ \Omega,\\
\tau(0)&=&\tau_0&\mathrm{in}\ \ \Omega,
\end{array}
\end{cases}
\end{eqnarray}
where $\Omega \subset \R^3$ is a domain, $T\in(0,\infty]$ and  $Re$ and  $We$ denote
the Reynolds and Weissenberg number of the fluid, respectively. Moreover, the term
$g_a$ is given by $g_a(\tau, \nabla u):=\tau W(u)-W(u)\tau-a\left(D(u)\tau+\tau D(u)\right)$
for some $a\in[-1,1]$, where $D(u)=1/2(\nabla u+(\nabla u)^\top)$ denotes the
deformation tensor and $W(u)=1/2(\nabla u-(\nabla u)^\top)$ the vorticity tensor,
respectively.

This set of equations originally was introduced by  J.G. Oldroyd
\cite{Oldroyd}, whose intention it was to describe mathematically
viscoelastic effects of certain types of fluids.

The study  of the above set of equations started by a pioneering
paper by Guillop\'e and Saut in 1990, see \cite{GS90a}, who proved
for the  situation of  {\em bounded domains} $\Omega \subset \R^3$,
the existence of a local, strong solution to equation \eqref{OB} in
suitable Sobolev spaces $H^s(\Omega)$. Moreover, this solution
exists on $[0,\infty)$ provided the data as well as the coupling
constant $\alpha$ between the two equations are sufficiently small.
For extensions to this results to the $L^p$-setting, see the work of
Fernand\'ez-Cara, Guill\'en and Ortega \cite{FGG98}.

The existence of global {\em weak solutions} in the case of
$\Omega=\R^n$ was proved by Lions and Masmoudi in \cite{LM00} for
$a=0$. For extensions of this result to scaling invariant spaces of
the form $L^\infty_{\loc}([0,T);H^s(\R^3))$ for $s>3/2$, we refer to
the work of Chemin and Masmoudi \cite{CM01}. An improvement of the
Chemin-Masmoudi blow-up criterion was presented recently by Lei,
Masmoudi and Zhou in \cite{LMZ10}.

The situation of {\em infinite Weissenberg numbers}, mainly using
the Lagrangian setting,  was considered for $\Omega=\R^3$ or for
bounded domains $\Omega \subset\R^3$ by Lin, Liu and Zhang in
\cite{LLZ05}, Lei, Liu and Zhou in \cite{LLZ08} and Lin and Zhang in
\cite{LZ08}. Further results, describing in particular the two
dimensional situation, can be found in in \cite{LZ05}, \cite{Lei06}
and \cite{Lei10}.

The situation of {\em exterior domains} was considered first in \cite{HNS12}. There, the
existence of a unique, global solution defined in certain function spaces was proved
provided the initial data as well as the coupling constant $\alpha$ are small enough.

In this paper we improve the main result given in \cite{HNS12} to
the situation of {\em non-small coupling coefficients} $\alpha$.
Indeed, given an exterior domain $\Omega \subset \R^3$ with smooth
boundary, we prove the existence of a unique, global solution to
equation (\ref{OB}) only under the assumption that the intial data
$u_0$ and $\tau_0$ are sufficiently small in their natural norms.
The idea of our approach is not completely new and goes back to the
work of Molinet and Talhouk \cite{MT04}. The main idea is to take
the divergence in  the third equation in (\ref{OB}) and to control
$\hp\dv\tau$ in the $H^1$-norm by the corresponding term in $L^2$
and by $\curl\dv\tau$, also measured in the $L^2$-norm. For doing
this, we need in particular to extend the well known estimate
 for
the $H^1$-norm of a function $u$ by the $L^2$-norms of $u,\dv u,
\curl u$  and by $u\cdot\nu$ in the $H^{1/2}$-norm for bounded
domains to the situation of unbounded domains. Note that the main
difficulty in the case of  exterior domains is due to the failure of
Poincar\'{e}'s inequality in this situation.  Hence, we must treat
the lower order terms of $u$ and $\partial_tu$ in a new way and some
new higher order energy estimates for $u$ and $\partial_tu$ need to
be developed as well.

These higher order energy estimates then imply that the local solution is satisfying a certain
differential inequality which in turn implies that the local solution of \eqref{OB} can be
extended for all positive times.

It is worthwhile to mention that also our approach to obtain local
solutions to \eqref{OB} seems to be new and quite different from the
ones known in the literature. In fact, we base our local existence
argument not on Schauder's fixed point theorem as it has been done
in almost all of the existing works, but on a variant of Banach's
contraction principle, avoiding hereby unnecessary uniqueness
arguments.

Some words about the derivation of equation \eqref{OB} are in order.
In fact, incompressible fluids are subject to the following system
of equations
\begin{eqnarray}\label{velocity}
\begin{cases}
\begin{array}{rll}
\partial_t u+(u\cdot\nabla)u&=&\dv \sigma,\\
\dv u&=&0,
\end{array}
\end{cases}
\end{eqnarray}
where $u, \sigma$ are the velocity and stress tensor, respectively. Moreover, $\sigma$ can be
decomposed into $-p Id+\tau$, where $p$ denotes the pressure of the fluid and $\tau$  the
tangential part of the stress tensor, respectively. For the Oldroyd-B model, $\tau$ is given by
the relation
\begin{equation}\label{tauequation}
\tau+\lambda_1\frac{D_a\tau}{Dt}=2\eta\left(D(u)+\lambda_2\frac{D_aD(u)}{Dt}\right),
\end{equation}
where $\frac{D_a}{Dt}$ denotes the ``objective derivative"
$$
\frac{D_a\tau}{Dt}=\partial_t\tau+(u\cdot\nabla)\tau+g_a(\tau,\nabla u),
$$
and $g_a(\tau, \nabla u):=\tau W(u)-W(u)\tau-a\left(D(u)\tau+\tau D(u)\right)$ for some
$a\in[-1,1]$. Here $D(u)$ is the deformation tensor defined as above and
$W(u)$ denotes the vorticity tensor. The parameters
$\lambda_1>\lambda_2>0$ denote  the relaxation and retardation time, respectively.

The tangential part of the stress tensor $\tau$ can be decomposed as $\tau=\tau_N+\tau_e$
where $\tau_N=2\eta\frac{\lambda_2}{\lambda_1}D(u)$ corresponds to the Newtonian part and
$\tau_e$ to the purely elastic part. Here $\eta$ denotes the fluid viscosity.
Denoting $\tau_e=\tau$, the above equations can be rewritten as
\begin{eqnarray}\label{utau}
\begin{cases}
\begin{array}{rll}
\partial_t u+(u\cdot\nabla)u-\eta(1-\alpha)\Delta u+\nabla p&=&\dv\tau,\\
\dv u&=&0,\\
\tau+\lambda_1\frac{D_a\tau}{Dt}&=&2\eta\alpha D(u),
\end{array}
\end{cases}
\end{eqnarray}
with the retardation parameter $\alpha:=1-\frac{\lambda_2}{\lambda_1}\in (0,1)$.

Using dimensionless variables, Oldroyd-B fluids may be thus described by the equations
\eqref{OB}.

In order to formulate our main result, let $A$ be the Stokes
operator in $L^2_\sigma(\Omega)$ defined by
$$
Au:=-\hp\Delta u\ \mbox{ for  all } u \in D(A):=H^2(\Omega)\cap H^1_0(\Omega) \cap
L^2_\sigma(\Omega),
$$
where the space $L^2_\sigma(\Omega)$ is defined precisely later in Section 2.
Moreover, we set $V:=H_0^1(\Omega)\cap L^2_\sigma(\Omega)$.

Our main results then reads  as follows.

\begin{theorem}\label{global}
Let $\Omega \subset\R^3$ be an exterior domain with boundary $\partial\Omega$ of class
$C^3$ and let $J=[0,\infty)$. Then there exists $\ve_0>0$ such that if
\begin{equation}\label{small}
\|u_0\|_{D(A)}+\|\tau_0\|_{H^2(\Omega)}\leq\ve_0,
\end{equation}
then equation (\ref{OB}) admits a unique, global strong solution $(u,p,\tau)$ for all
$t\in J$ satisfying
\begin{eqnarray*}
u &\in& C_b(J;D(A)) \mbox{ with } \nabla u\in L^2(J;H^2) \mbox{ and }
u' \in L^2(J;V) \cap C_b(J;L^2_\sigma(\Omega)),\\
\nabla p&\in& L^2(J;H^1(\Omega))\cap L^\infty(J;H^1(\Omega)),\\
\tau &\in& C_b(J;H^2(\Omega))\cap L^2(J;H^2(\Omega)) \mbox{ with }
\tau'\in C_b(J;H^1(\Omega))\cap L^2(J;L^2(\Omega)).
\end{eqnarray*}
\end{theorem}

\vn
\section{Preliminaries}
\noindent We start this section by recalling a higher order elliptic
regularity estimate for the stationary Stokes equations. A proof can
be found for example in \cite{Galdi}, Theorem \textbf{V} 4.7. In the
following, we denote by $\hat{H}^k(\Omega)$ the homogeneous Sobolev
spaces of order $k$.

\begin{lemma}\label{sS}
Let $m\in \{0, 1\}$ and $n\geq 3$. Assume that $\Omega\subset\R^n$ is an exterior domain with
boundary of class $C^{m+2}$ and $g\in H^m(\Omega)$. Then the equation
\begin{eqnarray}\label{eq1}
\left\{ \begin{array}{rl}
-\Delta u+\nabla p=g  & \mathrm{in}\ \ \Omega, \nonumber\\
\mathrm{div}u=0  & \mathrm{in}\ \ \Omega,\\
u=0 & \mathrm{on}\ \ \partial \Omega \nonumber
 \end{array}
 \right.
\end{eqnarray}
admits a solution $(u,p)\in\hat{H}^{m+2}(\Omega)\times\hat{H}^{m+1}(\Omega)$ which is unique
provided $\nabla u\in L^2(\Omega)$. In this case, there exists a constants $C>0$ such that
\begin{eqnarray}\label{sS_L^2}
\|\nabla^2 u\|_{H^m}+\|\nabla p\|_{H^m}\leq C(\|g\|_{H^m}+\|\nabla u\|_{L^2}).
\end{eqnarray}
\end{lemma}

\vn
The following variant of Lemma \ref{sS} will be important in the sequel.
We first introduce the spaces
$$
G_2(\Omega):=\{u\in L^2(\Omega): u=\nabla\pi \mbox{ for  some } \pi\in H^1_{loc}(\Omega)\},
$$
and
$$
L^2_\sigma(\Omega):=\{u\in C_c^\infty(\Omega): \mbox{div }u=0
\mbox{ in } \Omega\}^{\overline{\mbox{\,\,\, }}|_{\|\cdot\|_{L^2}}}
$$
Then $L^2(\Omega)$ can be decomposed into
$$
L^2(\Omega)=L^2_\sigma(\Omega)\oplus G_2(\Omega),
$$
and there exists a unique projection $\hp:\ L^2(\Omega)\rightarrow
L^2_\sigma(\Omega)$ with $G_2(\Omega)$ as its null space. $\hp$ is
called the Helmholtz projection. It is well known that
$$
L^2_\sigma(\Omega)=\{u\in L^2(\Omega): \dv u=0, \nu\cdot u|_{\partial\Omega}=0\},
$$
where $\nu$ is the exterior normal to $\partial\Omega$.

\begin{corollary}\label{psS}
Assume that the assumptions of Lemma \ref{sS} hold. Then there exists a constant $C>0$ such
that the solution $(u,p)$ of equation \eqref{eq1} satisfies
\begin{eqnarray}\label{psS_L^2}
\|\nabla^2 u\|_{H^m}\leq C(\|\hp g\|_{H^m}+\|\nabla u\|_{L^2}).
\end{eqnarray}
\end{corollary}

\begin{proof}
Applying the Helmholtz decomposition to $g$ and noting that $\hp$ acts as a bounded operator
on $H^1(\Omega)$ yields $g=\hp g + \nabla \varphi$ for some $\nabla\varphi\in H^m(\Omega)$.
Hence, the first line of the above Stokes equation may be rewritten as
$-\Delta u+\nabla\tilde{p}=\hp g$, where $\tilde{p}=p-\varphi$. Thus, estimate
(\ref{psS_L^2}) follows immediately from eestimate (\ref{sS_L^2}).
\end{proof}

\begin{remark}\label{AsS}{\rm
Rewriting the above Stokes equation \eqref{eq1} as  $Au=\hp g$, we may replace the term  $\hp g$
in \eqref{psS_L^2} by $Au$.
}\end{remark}

We next recall a well known estimates for the $H^1$-norm of a function  defined on a bounded
domain $\Omega \subset \R^3$.

\begin{lemma}[\cite{Girault}]\label{lemmab}
Let $\Omega\subset \R^3 $ be a bounded domain with  boundary $\partial\Omega$ of class $C^2$.
Then there exists a constant $C>0$ such that
\begin{equation*}\label{binq}
\|u\|_{H^m(\Omega)}\leq C \left(\|u\|_{L^2(\Omega)} + \|\dv u\|_{L^2(\Omega)}+
\|\curl u\|_{L^2(\Omega)}+\|u\cdot\nu\|_{H^{\frac12}(\partial\Omega)}\right),
\quad u\in H^1(\Omega)^3.
\end{equation*}
\end{lemma}

The following variant of Lemma \ref{lemmab} concerns exterior
domains. More precisely, the following proposition holds true.

\begin{proposition}\label{propunb}
Let $\Omega\subset \R^3$ be an exterior domain with  boundary
$\partial\Omega$ of class $C^2$. Then there exists a constant $C>0$
such that
\begin{equation*}\label{unbinq}
\|u\|_{H^1(\Omega)}\leq C\left(\|u\|_{L^2(\Omega)}+ \|\dv u\|_{L^2(\Omega)}+
\|\curl u\|_{L^2(\Omega)}+
\|u\cdot\nu\|_{H^{\frac12}(\partial\Omega)}\right),\quad u\in H^1(\Omega)^3.
\end{equation*}
\end{proposition}

\begin{proof}
We establish first a similar inequality for functions defined on  $\R^3$. Assuming,
for the time being,  that $u$ is smooth, we obtain for the $j$-th component $u^j$ of $u$
$$
\Delta u^j=\sum_{1\leq i\leq3}\partial_i(\partial_iu^j-\partial_ju^i)+\partial_j\dv u =
\sum_{1\leq i\leq3}\partial_i(\curl u)^{ij} + \partial_j\dv u.
$$
Hence,
$$
 u^j=-\sum_{1\leq i\leq3}\partial_i(-\Delta)^{-1}(\curl u)^{ij}-\partial_j(-\Delta)^{-1}\dv u,
$$
and
\begin{eqnarray*}
  \partial_ku^j&=&-\sum_{1\leq i\leq3}\partial_k\partial_i(-\Delta)^{-1}(\curl u)^{ij}
-\partial_j\partial_k(-\Delta)^{-1}\dv u\\
 &=&-\sum_{1\leq i\leq3}R_iR_k(\curl u)^{ij}-R_jR_k\dv u,
 \end{eqnarray*}
 where $R_j:=\partial_j(-\Delta)^{-\frac{1}{2}}$ denotes the $j$-th Riesz transforms.
Classical results on the boundedness of the Riesz transforms imply
 \begin{equation}\label{whole}
 \|\nabla u\|_{L^2}\leq C\left(\|\dv u\|_{L^2}+\|\curl u\|_{L^2}\right)
 \end{equation}
for some $C>0$. By density, this estimate transfers to functions $u$ belonging to
$H^1(\R^3)$.

Next, our aim is to combine the above  estimate with the
corresponding one in bounded domains in order to obtain the
assertion for exterior domains. To this end, let $R>0$ such that
$\Omega^c\subset B_R(0):=\{x\in\mathbb{R}^3:|x|<R\}$ and set
$$
D:=\Omega\cap B_{R+3}(0), \ \ \Gamma_1:=\partial\Omega\ \ \mathrm{and}\ \
\Gamma_2:=\partial D\backslash\Gamma_1.
$$
Next, we choose a cut-off function $\phi\in C^{\infty}_c(B_{R+3}(0))$ such that
$0\leq\phi\leq1$ and
$$
 \phi(x)=
 \begin{cases}
 1,\ \ |x|\leq R+1,\\
 0,\ \ |x|\geq R+2.
 \end{cases}
 $$
and decompose $u$ as
$$
u=\phi u+(1-\phi)u=:u_1+u_2.
$$
Clearly, $u_1\in H^1(D)$ and $u_1\equiv0$ near $\Gamma_2$. Moreover, $u_2\in H^1_0(\Omega)$ and
after zero extension (still denoted by $u_2$), we can regard $u_2$ as an element in
$H^1(\mathbb{R}^3)$. Now it follows from (\ref{binq}) and (\ref{whole}) that
\begin{eqnarray*}
\|\nabla u\|_{L^2(\Omega)}&\leq&\|\nabla u_1\|_{L^2(D)}+\|\nabla u_2\|_{L^2(\mathbb{R}^3)}\\
&\leq&C[\|\phi u\|_{L^2(D)}+\|\dv(\phi u)\|_{L^2(D)}+\|\curl(\phi u)\|_{L^2(D)}+
\|(\phi u)\cdot\nu\|_{H^{\frac12}(\Gamma_1)}\\
&& +\|(\phi u)\cdot\nu\|_{H^{\frac12}(\Gamma_2)}+ \|\dv((1-\phi) u)\|_{L^2(\mathbb{R}^3)}
+\|\curl((1-\phi) u)\|_{L^2(\mathbb{R}^3)}]\\
&\leq& C[\|u\cdot\nu\|_{H^{\frac12}(\Gamma_1)}+ \|\phi u\|_{L^2(D)}+\|\nabla\phi\cdot
u\|_{L^2(D)}+\|\nabla\phi\otimes u-u\otimes\nabla \phi\|_{L^2(D)}\\
&&+ \|\phi \dv u\|_{L^2(D)}+\|(1-\phi)\dv u\|_{L^2(\mathbb{R}^3)}+\|\phi\curl u\|_{L^2(D)} +
\|(1-\phi)\curl u\|_{L^2(\mathbb{R}^3)}]\\
&\leq& C[\|u\cdot\nu\|_{H^{\frac12}(\Gamma_1)}+\|u\|_{L^2(\Omega)}+\| \dv u\|_{L^2(\Omega)} +
\|\curl u\|_{L^2(\Omega)}],
 \end{eqnarray*}
and the estimate (\ref{unbinq}) follows immediately.
\end{proof}

In order to construct a local solution to (\ref{OB}), we  study
first two linearized equations; the first one for the velocity $u$
and the second one for the tangential part of the stress tensor
$\tau$, respectively.

First, given $T>0$, we recall some results on the Stokes equation
\begin{eqnarray}\label{S}
\begin{cases}
\begin{array}{rll}
\partial_tu-\Delta u+\nabla p&=f& \mathrm{in} \ \ \Omega\times(0, T),\\
\dv u&=0&\mathrm{in} \ \ \Omega\times(0, T),\\
u&=0& \mathrm{on} \ \ \partial\Omega\times(0, T),\\
u(0)&=u_0& \mathrm{in} \ \ \Omega,
\end{array}
\end{cases}
\end{eqnarray}
where $f$ is a given external force.

\begin{proposition}[\cite{Talhouk}]\label{lem2.1}
Let $\Omega\subset\mathbb{R}^3$ be an exterior  domain with boundary
$\partial\Omega$ of class $C^3$. Assume that $ f\in L^2(0,T;H^1(\Omega)),
f'\in L^2(0,T;H^{-1}(\Omega))$ and $u_0\in H^2(\Omega)\cap
V(\Omega)$. Then there exists a unique solution $(u,p)$ of equation (\ref{S})
satisfying
\begin{eqnarray*}
u&\in& L^2(0,T;H^3(\Omega))\cap C([0,T];H^2(\Omega)\cap V(\Omega)),
u'\in L^2(0,T;V(\Omega))\cap C([0,T];L^2_\sigma(\Omega)),\\
p&\in& L^2(0,T;H^2_{\textrm{loc}}(\Omega)).
\end{eqnarray*}
Moreover, there exists a constant $C>0$ such that
\begin{eqnarray*}
\|u\|^2_{L^2(H^3)\cap L^\infty(H^2\cap V)}+\|u'\|^2_{L^2(V)\cap
L^\infty(L^2_\sigma)}+\|\nabla p\|_{L^2(H^1)}^2 &\leq& C
[\|u_0\|_{H^2}^2+\|f(0)\|_{L^2}+\|f\|_{L^2(H^1)}  \\
&& + \|f'\|_{L^2(H^{-1})}].
\end{eqnarray*}
\end{proposition}

\vn
Next, consider the  transport equation
\begin{eqnarray}\label{T}
\begin{cases}
\begin{array}{rlll}
\We(\partial_t\tau+(v\cdot\nabla)\tau)+\tau&=&2\alpha D(v) - \We g_a(\tau,\nabla v),
&\textrm{in} \ \  \Omega\times(0, T),\\
\tau(0)&=&\tau_0, &\mathrm{in} \ \ \Omega,
\end{array}
\end{cases}
\end{eqnarray}
where $v$ is a given velocity field.

\vn
\begin{proposition}[\cite{Talhouk}]\label{lem2.2}
Let $\Omega\subset\mathbb{R}^3$ be an exterior  domain with boundary
$\partial\Omega$ of class $C^3$. Assume that $v\in L^1(0,T;H^3(\Omega)\cap
V(\Omega))$ and $ \tau_0\in H^2(\Omega)$. Then there exists a unique
solution of equation (\ref{T}) and a  constant
$C>0$ such that
\begin{equation*}
\|\tau\|_{L^\infty(0,T;H^2(\Omega))}\leq (\|\tau_0\|_{H^2}+\frac{2\alpha}{C\We})
\exp(C\|v\|_{L^1(H^3)}).
\end{equation*}
If, in addition, $v\in C([0,T];H^2(\Omega)\cap V(\Omega))$, then
$\tau'\in C([0,T];H^1(\Omega))$ and
\begin{equation*}
\|\tau'\|_{L^\infty(0,T;H^1)}\leq
C(\|v\|_{L^\infty(H^2)}+\frac{1}{C\We})(\|\tau_0\|_{H^2}+\frac{2\alpha}{C\We})
\exp(C\|v\|_{L^1(H^3)}).
\end{equation*}
\end{proposition}

The assertions of Propositions \ref{lem2.1} and \ref{lem2.2} are stated e.g. in
\cite{Talhouk} even for a more general class of domains, however, without giving a proof.

\section{Existence and Uniqueness of a Local Solution}

\noindent
In this section we prove that the system \eqref{OB} possesses a  unique, local solution provided
the intial data are smooth enough. More precisely, the following result holds true.

\begin{proposition}\label{local}
Assume that $\Omega\subset \mathbb{R}^3$ is an exterior domain with boundary $\partial\Omega$
of class $C^3$. Let  $u_0\in D(A)$ and $\tau_0\in H^2(\Omega)$. Then there exist
$T_\star>0$ and functions
\begin{eqnarray*}
u&\in& L^2(0,T_\star;H^3(\Omega))\cap C([0,T_\star];D(A)) \mbox{ with } u'\in L^2(0,T_\star,V)
\cap C([0,T_\star];L^2_\sigma(\Omega)),\\
p&\in& L^2(0,T;H^2_{loc}(\Omega)) \mbox{ with } \nabla p\in L^2(0,T_\star;H^1(\Omega)),\\
\tau&\in& C([0,T_\star];H^2(\Omega)) \mbox{ with } \tau' \in C([0,T_\star], H^1(\Omega))
\end{eqnarray*}
such that $(u,p,\tau)$ is the unique solution to equation (\ref{OB}) on $(0,T_\star)$.
\end{proposition}

Let us begin the proof of Proposition \ref{local} with the following  variant of
Banach's fixed point theorem. For a proof, we refer e.g. to \cite{Kreml}.

\begin{lemma}[\cite{Kreml}]\label{fix}
Let $X$ be a reflexive Banach space or let $X$ have a separable pre-dual. Let $K$ be a
convex, closed and bounded subset of $X$ and assume that  $X$ is embedded into a
Banach space $Y$. Let $\Phi: X\rightarrow X$ map $K$ into $K$ and
assume there exists  $q<1$ such that
$$
\|\Phi(x)-\Phi(x)\|_{Y}\leq q \|x-y\|_Y,\ \ x,y\in K.
$$
Then there exists a unique fixed point of $\Phi$ in $K$.
\end{lemma}

\vn
Our proof of Proposition \ref{local} relies on a combination of
Propositions \ref{lem2.1} and \ref{lem2.2} with Lemma \ref{fix}.
To this end, consider for  $T>0$ the following function spaces
\begin{eqnarray*}
E_1 &:=& L^2(0,T;H^3(\Omega))\cap L^\infty(0,T;H^2(\Omega)\cap V),\\
E_2 &:=& L^2(0,T;V(\Omega))\cap L^\infty(0,T;L^2_\sigma(\Omega)), \\
F_1 &:=& L^\infty(0,T;H^2(\Omega)), \\
F_2 &:=& L^\infty(0,T;H^1(\Omega))),
\end{eqnarray*}
and for $B_1,B_2>0$ define the set $K(T)$ by
\begin{eqnarray*}
K(T)&:=&\{(v,\theta) \in E_1 \times F_1,  v' \in E_2, \theta' \in F_2,
v(0)= u_0, \theta(0) = \tau_0 \mbox{ and } \\
&& \quad \|v\|^2_{E_1} + \|v'\|^2_{E_2} \leq B_1,
\|\theta\|_{F_1}\leq B_1, \|\theta'\|_{F_2}\leq B_2\}.
\end{eqnarray*}

Next, given $(v, \theta)\in K(T)$, we define the mapping
$$
\Phi(v,\theta):=(u,\tau),
$$
where $(u,\tau)$ is defined to be the unique solution of the
corresponding linearized problem of (\ref{OB})
\begin{eqnarray}\label{LOB}
\begin{cases}
\begin{array}{rlll}
\Rey\partial_t u+(1-\alpha)Au& =& -\hp\dv(v\otimes v)+\hp\dv\theta&\mathrm{in}\ \
\Omega\times(0,T),\\
\We(\partial_t\tau+(v\cdot\nabla)\tau)+\tau& =& 2\alpha D(v)-\We g_a(\tau,\nabla v)
&\mathrm{in}\ \ \Omega\times(0,T),\\
u&=&0&\mathrm{on}\ \ \partial\Omega\times(0,T),\\
u(0)&=&u_0&\mathrm{in}\ \ \Omega,\\
\tau(0)&=&\tau_0&\mathrm{in}\ \ \Omega,
\end{array}
\end{cases}
\end{eqnarray}
where $A$ denotes the Stokes operator defined as in Section 1. It
follows from Proposition \ref{lem2.1} and \ref{lem2.2} that for
appropriate choices of $B_1$ and $B_2$, there exists $T_1>0$ such
that $\Phi(K(T_1))\subset K(T_1)$.

Next, we will prove that there exists $T_\star \in (0,T_1]$
such that $\Phi$ is contractive on $Y(T_\star)$, where $Y(T)$ is
defined by
\begin{eqnarray*}
Y(T):=\left\{(v, \theta) \in  L^\infty(0,T;L^2(\Omega)) \times L^\infty(0,T;L^2(\Omega)),
\nabla v\in L^2(0,T;L^2(\Omega))\right\}.
\end{eqnarray*}
Indeed, for $(v_i,\theta_i)\in K(T_1)$ let $(u_i,\tau_i)=\Phi(v_i,\theta_i)$ for  $i=1,2$.
Moreover, we set $\bar{u}=u_1-u_2$ and $\bar{\tau}=\tau_1-\tau_2$.
Then $(\bar{u},\bar{\tau})$ satisfies the equation
\begin{eqnarray}\label{DLOB}
\qquad\quad
\begin{cases}
\begin{array}{rlll}
\Rey\partial_t\bar{u}+(1-\alpha)A\bar{u}&=&-\mathbb{P}\dv(\bar{v}\otimes v_1+v_2\otimes\bar{v} )
+\mathbb{P}\dv\bar{\theta}&\mathrm{in}\ \ \Omega\times(0,T),\\
\We\partial_t\bar{\tau}+\bar{\tau}&=&2\alpha
D(\bar{v})-\We[(\bar{v}\cdot\nabla)\tau_1+(v_2\cdot\nabla)\bar{\tau} + & \\
&& g_a(\bar{\tau},\nabla v_1)+g_a(\tau_2, \nabla \bar{v})]&\mathrm{in}\ \
\Omega\times(0,T),\\
\bar{u}&=&0&\mathrm{on}\ \ \partial\Omega\times(0,T),\\
\bar{u}(0)&=&0&\mathrm{in}\ \ \Omega,\\
\bar{\tau}(0)&=&0&\mathrm{in}\ \ \Omega.
\end{array}
\end{cases}
\end{eqnarray}

Taking the $L^2$ inner product of $(\ref{DLOB})_1$ with $\bar{u}$,
we obtain
\begin{eqnarray*}
\frac{1}{2}\frac{d}{dt}(\Rey\|\bar{u}\|_{L^2}^2)+(1-\alpha)\|\nabla\bar{u}\|_{L^2}^2&=&
\Rey(\bar{v}\otimes
v_1+v_2\otimes\bar{v}|\nabla\bar{u})-(\bar{\theta}|\nabla\bar{u})\\
&\leq&\frac{1-\alpha}{2}\|\nabla\bar{u}\|_{L^2}^2+\frac{2\Rey^2}{1-\alpha}(\|v_1\|_{L^\infty}^2
+\|v_2\|_{L^\infty}^2)
\|\bar{v}\|_{L^2}^2+\frac{1}{1-\alpha}\|\bar{\theta}\|_{L^2}^2\\
&\leq&\frac{1-\alpha}{2}\|\nabla\bar{u}\|_{L^2}^2+\frac{C\Rey^2}{1-\alpha}
(\|v_1\|_{H^2}^2+\|v_2\|_{H^2}^2)
\|\bar{v}\|_{L^2}^2+\frac{1}{1-\alpha}\|\bar{\theta}\|_{L^2}^2.
\end{eqnarray*}
Consequently,
\begin{eqnarray}\label{l2u}
\frac{d}{dt}(\Rey\|\bar{u}\|_{L^2}^2)+(1-\alpha)\|\nabla\bar{u}\|_{L^2}^2
\leq\frac{C\Rey^2}{1-\alpha}(\|v_1\|_{H^2}^2+\|v_2\|_{H^2}^2)
\|\bar{v}\|_{L^2}^2+\frac{2}{1-\alpha}\|\bar{\theta}\|_{L^2}^2.
\end{eqnarray}
Taking the $L^2$ inner product of $(\ref{DLOB})_2$ with $\bar{\tau}$, we are led to
\begin{eqnarray}\label{l2t}
\nonumber\frac{1}{2}\frac{d}{dt}(\We\|\bar{\tau}\|_{L^2}^2)+\|\bar{\tau}\|_{L^2}^2
&=&2\alpha(D(\bar{v})|\bar{\tau})
-\We((\bar{v}\cdot\nabla)\tau_1+g_a(\bar{\tau},\nabla v_1)+g_a(\tau_2, \nabla \bar{v})|
\bar{\tau})\\
\nonumber&\leq&\frac\delta4\|\nabla\bar{v}\|_{L^2}^2+\frac{4\alpha^2}{\delta}\|
\bar{\tau}\|_{L^2}^2
+\, C \We(\|\bar{v}\|_{L^6}\|\nabla\tau_1\|_{L^3}+\|\nabla\bar{v}\|_{L^2}\|\tau_2\|_{L^\infty})
\|\bar{\tau}\|_{L^2}\\
\nonumber&& + C \We\|\nabla v_1\|_{L^\infty}\|\bar{\tau}\|_{L^2}^2\\
\nonumber&\leq& \frac\delta4\|\nabla\bar{v}\|_{L^2}^2+\frac{4\alpha^2}{\delta}\|
\bar{\tau}\|_{L^2}^2 +
C \We\|\nabla\bar{v}\|_{L^2}(\|\tau_1\|_{H^2}+\|\tau_2\|_{H^2})
\|\bar{\tau}\|_{L^2}\\
\nonumber&&+C\We\|\nabla v_1\|_{H^2}\|\bar{\tau}\|_{L^2}^2\\
&\leq&\frac\delta2\|\nabla\bar{v}\|_{L^2}^2+
\frac{C_1}{2}\big(\frac{1}{\delta}(1+\|\tau_1\|_{H^2}^2+\|\tau_2\|_{H^2}^2)+\|\nabla
v_1\|_{H^2}\big)\We\|\bar{\tau}\|_{L^2}^2,
\end{eqnarray}
for some $C_1>0$ and all $\delta>0$. It follows from (\ref{l2u}) and (\ref{l2t}) that,
for all $t\in [0,T]$,
\begin{eqnarray*}
\Rey\|\bar{u}\|_{L^2}^2 &+&\We\|\bar{\tau}\|_{L^2}^2+\int_0^t\left((1-\alpha)\|
\nabla\bar{u}\|_{L^2}^2
+2\|\bar{\tau}\|_{L^2}^2\right)ds\\
&\leq&TC_2(1+\|v_1\|_{L^\infty(H^2)}^2+\|v_2\|_{L^\infty(H^2)}^2)
(\|\bar{v}\|_{L^\infty(L^2)}^2+\|\bar{\theta}\|_{L^\infty(L^2)}^2)+
\delta\int_0^t\|\nabla\bar{v}\|_{L^2}^2ds\\
&&+C_1\int_0^t (\frac{1}{\delta}(1+\|\tau_1\|_{H^2}^2+\|\tau_2\|_{H^2}^2)+\|\nabla
v_1\|_{H^2})(\We\|\bar{\tau}\|_{L^2}^2)ds\\
&\leq&TC_2(1+2B_1)(\|\bar{v}\|_{L^\infty(L^2)}^2+\|\bar{\theta}\|_{L^\infty(L^2)}^2)+
\delta\int_0^t\|\nabla\bar{v}\|_{L^2}^2ds\\
&&+C_1\int_0^t(\frac{1}{\delta}(1+\|\tau_1\|_{H^2}^2+\|\tau_2\|_{H^2}^2)+\|\nabla
v_1\|_{H^2})(\We\|\bar{\tau}\|_{L^2}^2)ds,
\end{eqnarray*}
for some $C_2>0$. Gronwall's inequality implies then
\begin{eqnarray*}
\Rey\|\bar{u}\|_{L^\infty(L^2)}^2 &+&\We\|\bar{\tau}\|_{L^\infty(L^2)}^2 +
\int_0^T\left((1-\alpha)\|\nabla\bar{u}\|_{L^2}^2
+2\|\bar{\tau}\|_{L^2}^2\right)ds\\
&\leq&[TC_2(1+2B_1)(\|\bar{v}\|_{L^\infty(L^2)}^2+\|\bar{\theta}\|_{L^\infty(L^2)}^2)+
\delta\int_0^T\|\nabla\bar{v}\|_{L^2}^2ds]\\
&&\times [1+C_1(\frac{T}{\delta}(1+\|\tau_1\|_{L^\infty(H^2)}^2+
\|\tau_2\|_{\infty(H^2)}^2)
+\sqrt{T}\|\nabla v_1\|_{L^2(H^2)})\\
&&\times\exp\big(C_1(\frac{T}{\delta}(1+\|\tau_1\|_{L^\infty(H^2)}^2
+\|\tau_2\|_{\infty(H^2)}^2)
+\sqrt{T}\|\nabla v_1\|_{L^2(H^2)})\big)]\\
&\leq&[TC_2(1+2B_1)(\|\bar{v}\|_{L^\infty(L^2)}^2+\|\bar{\theta}\|_{L^\infty(L^2)}^2)+
\delta\int_0^T\|\nabla\bar{v}\|_{L^2}^2ds]\\
&&\times [1+C_1\big(\frac{T}{\delta}(1+2B_1^2)
+\sqrt{TB_1}\big)\exp\big(C_1(\frac{T}{\delta}(1+2B_1^2)
+\sqrt{TB_1})\big)],
\end{eqnarray*}
Setting  $\delta=\min\{\Rey,\We,1-\alpha\}(4+8C_1\exp(2C_1))^{-1}$ and
$T_\star=\min\{T_1, \frac{\delta}{1+2B_1^2}, \frac{1}{B_1},
\frac{\min\{\Rey, \We, 1-\alpha\}}{4C_2(1+2B_1)(1+2C_1\exp(2C_1))}\}$,
we see that for all $T\leq T_\star$,
\begin{eqnarray*}
\|\bar{u}\|_{L^\infty(L^2)}^2+\|\bar{\tau}\|_{L^\infty(L^2)}^2+\int_0^T
\left(\|\nabla\bar{u}\|_{L^2}^2
+\|\bar{\tau}\|_{L^2}^2\right)ds
\leq\frac{1}{4}\big(\|\bar{v}\|_{L^\infty(L^2)}^2+\|\bar{\theta}\|_{L^\infty(L^2)}^2
+\int_0^T\|\nabla\bar{v}\|_{L^2}^2ds\big).
\end{eqnarray*}
Hence, $\Phi$ is contractive as a mapping from  $Y(T_\star)$ to $Y(T_\star)$.
The assertion of Proposition \ref{local} thus follows from Lemma \ref{fix}.

\rightline{$\Box$}

\section{Proof of the Main Theorem}

Let $(u,p,\tau)$ be the local solution to equation \ref{OB} constructed in Proposition
 \ref{local}. We recall from this proposition that
\begin{eqnarray*}
u &\in& L^2(0,T_\star;H^3)\cap C([0,T_\star];D(A)) \mbox{ with } u' \in
L^2(0,T_\star;V)\cap C([0,T_\star];L^2_\sigma) \mbox{ and} \\
\tau &\in& C([0,T_\star];H^2) \mbox{ with } \tau' \in C([0,T_\star];H^1).
\end{eqnarray*}
Our proof for the existence of a unique, global solution to \eqref{OB} is based on the following
a priori estimates for $u, \tau, u'$ and $\tau'$.

Let us begin with an a priori estimate for $\tau$. To this end, we take the inner product
of $(\ref{OB})_3$ with $\tau$ and obtain

\begin{equation}\label{tau_L^2}
\frac{\We}{2}\frac{d}{dt}\|\tau\|_{L^2}^2+\|\tau\|_{L^2}^2
=2\alpha(D(u)|\tau)-\We(g_a(\tau, \nabla u)|\tau)
\leq 2\alpha\|\nabla u\|_{L^2}\|\tau\|_{L^2}+C \We\|\nabla u\|_{H^2}\|\tau\|_{L^2}^2.
\end{equation}

Similarly,
\begin{eqnarray*}
\frac{\We}{2}\frac{d}{dt}\|\nabla\tau\|_{L^2}^2+\|\nabla\tau\|_{L^2}^2
&=&2\alpha(\nabla D(u)|\nabla\tau)-\We(\nabla g_a(\tau, \nabla u)|\nabla\tau)-
\We(\partial_lu^k\partial_k\tau^{ij}|\partial_l\tau^{ij})\\
&\leq&2\alpha\|\nabla^2 u\|_{L^2}\|\nabla\tau\|_{L^2}+C \We\|\nabla
u\|_{H^2}\|\nabla\tau\|_{L^2}^2+C \We\|\nabla^2 u\|_{L^2}\|\tau\|_{L^\infty}\|\nabla\tau\|_{L^2},
\end{eqnarray*}
and for $i,j,k,l,m \in \{1,2,3\}$
\begin{eqnarray*}
\frac{\We}{2}\frac{d}{dt}\|\nabla^2\tau\|_{L^2}^2 +\|\nabla^2\tau\|_{L^2}^2
&=&2\alpha(\nabla^2 D(u)|\nabla^2\tau) - \We(\partial_{lm} g_a(\tau, \nabla u)^{ij}|
\partial_{lm}\tau^{ij})-\We(\partial_{lm}u^k\partial_k\tau^{ij}|\partial_{lm}\tau^{ij})\\
&& - \We(\partial_mu^k\partial_{kl}\tau^{ij}|\partial_{lm}\tau^{ij}) -
\We(\partial_lu^k\partial_{km}\tau^{ij}|\partial_{lm}\tau^{ij})\\
&\leq& 2\alpha\|\nabla^3 u\|_{L^2}\|\nabla^2\tau\|_{L^2}
 +C \We\|\nabla u\|_{H^2}\|\nabla^2\tau\|_{L^2}^2 + \\
&& C \We\|\nabla^3 u\|_{L^2}\|\tau\|_{L^\infty}\|\nabla^2\tau\|_{L^2}
+ C \We \|\nabla^2 u\|_{H^1}\|\nabla\tau\|_{H^1}\|\nabla^2\tau\|_{L^2}.
\end{eqnarray*}
Combing the above three inequalities, we obtain
\begin{eqnarray*}
\frac{\We}{2}\frac{d}{dt}\|\tau\|_{H^2}^2+\|\tau\|_{H^2}^2
&\leq&2\alpha\|\nabla u\|_{H^2}\|\tau\|_{H^2}+C \We\|\nabla u\|_{H^2}\|\tau\|_{H^2}^2\\
&\leq&\frac{1}{2}\|\tau\|^2_{L^2}+C\alpha^2\|\nabla u\|_{H^2}^2+C \frac{\We^2}{\alpha^2}
\|\tau\|_{H^2}^4,
\end{eqnarray*}
and thus
\begin{equation}\label{tau_H^2}
\We\frac{d}{dt}\|\tau\|_{H^2}^2+\|\tau\|_{H^2}^2
\leq C\alpha^2\|\nabla u\|_{H^2}^2+C \frac{\We^2}{\alpha^2}\|\tau\|_{H^2}^4.
\end{equation}

Before estimating $u$, let us first apply the Helmholz projection  $\mathbb{P}$ to
$(\ref{OB})_1$. This yields
\begin{equation}\label{Pu}
\Rey(\partial_tu+\mathbb{P}((u\cdot\nabla)u))+(1-\alpha)Au=\mathbb{P}\dv\tau.
\end{equation}
Next, we estimate the term $\|\nabla u\|_{H^2}$ appearing in the right hand side
of (\ref{tau_H^2}). Corollary \ref{psS} and Remark \ref{AsS} imply that
\begin{equation}\label{nabla u_H^1}
\|\nabla^2 u\|_{H^1}\leq C\left(\|Au\|_{H^1}+\|\nabla u\|_{L^2}\right).
\end{equation}
We deduce from  equation (\ref{Pu}) that
\begin{eqnarray}\label{nabla u_H^1+}
\|\nabla Au\|_{L^2}&\leq&\frac{\Rey}{1-\alpha}\|\nabla u_t\|_{L^2} +
\frac{\Rey}{1-\alpha}\|\nabla\mathbb{P}((u\cdot\nabla)u)\|_{L^2} +
\frac{1}{1-\alpha}\|\nabla\mathbb{P}\dv\tau\|_{L^2}\nn \\
&\leq&\frac{\Rey}{1-\alpha}\|\nabla
u_t\|_{L^2}+\frac{\Rey}{1-\alpha}\|(u\cdot\nabla)u\|_{H^1} +
\frac{1}{1-\alpha}\|\nabla\mathbb{P}\dv\tau\|_{L^2}
\end{eqnarray}
By the  Gagliardo-Nirenberg as well as by Sobolev's inequality, we have
\begin{equation}\label{GN}
\|u\|_{L^\infty}\leq C\|u\|_{L^6}^{\frac{1}{2}}\|\nabla u\|_{L^6}^{\frac{1}{2}}
\leq C\|\nabla u\|_{L^2}^{\frac{1}{2}}\|\nabla u\|_{H^1}^{\frac{1}{2}},
\end{equation}
which allows us to bound the term  $\|(u\cdot\nabla)u\|_{H^1}$ as
\begin{eqnarray}\label{nabla u_H^1++}
\|(u\cdot\nabla)u\|_{H^1}&=&\|\nabla((u\cdot\nabla)u)\|_{L^2}+\|(u\cdot\nabla)u\|_{L^2}\nn \\
&\leq&\|\nabla u\|^2_{L^4}+\|u\|_{L^\infty}\|\nabla^2u\|_{L^2}+\|u\|_{L^6}\|\nabla u\|_{L^3}
\nn\\
&\leq&C\left(\|\nabla u\|_{H^1}^2 +
\|u\|_{L^6}^{\frac{1}{2}}\|\nabla u\|_{L^6}^{\frac{1}{2}}\|\nabla^2u\|_{L^2}
+ \|\nabla u\|_{L^2}^{\frac{3}{2}}\|\nabla u\|_{H^1}^{\frac{1}{2}}\right)\nn\\
&\leq&C\left(\|\nabla u\|_{H^1}^2+\|\nabla u\|_{L^2}^{\frac{1}{2}}
\|\nabla u\|_{H^1}^{\frac{3}{2}}+\|\nabla u\|_{L^2}^{\frac{3}{2}}
\|\nabla u\|_{H^1}^{\frac{1}{2}}\right)\nn\\
&\leq&C\|\nabla u\|_{H^1}^2.
\end{eqnarray}
Recalling Corollary \ref{psS} and Remark \ref{AsS}, we infer that
\begin{equation}\label{nabla^2uL^2}
\|\nabla u\|_{H^1}\leq C\left(\|Au\|_{L^2}+\|\nabla u\|_{L^2}\right).
\end{equation}
Combing now the estimates (\ref{nabla u_H^1})--(\ref{nabla u_H^1++}) with (\ref{nabla^2uL^2})
yields
\begin{eqnarray}\label{nabla u_H^2}
\nonumber\|\nabla u\|_{H^2}^2&\leq&C\big(\|Au\|_{L^2}^2 + \|\nabla
u\|_{L^2}^2+\frac{\Rey^2}{(1-\alpha)^2}\|\nabla u_t\|_{L^2}^2 +
\frac{1}{(1-\alpha)^2}\|\nabla\mathbb{P}\dv\tau\|_{L^2}^2\\
&&+\frac{\Rey^2}{(1-\alpha)^2}\|Au\|_{L^2}^4 +
\frac{\Rey^2}{(1-\alpha)^2}\|\nabla u\|_{L^2}^4\big).
\end{eqnarray}
Finally,  estimate (\ref{tau_H^2}) combined with estimate (\ref{nabla u_H^2}) implies that
\begin{eqnarray}\label{tau_H^22}
\We\frac{d}{dt}\|\tau\|_{H^2}^2+\|\tau\|_{H^2}^2+2\|\nabla u\|_{H^2}^2
&\leq&\kappa_1\left(\|Au\|_{L^2}^2+\|\nabla u\|_{L^2}^2+\|\nabla u_t\|_{L^2}^2 +
\|\nabla\mathbb{P}\dv\tau\|_{L^2}^2\right)\nn\\
&& + C \left(\|Au\|_{L^2}^4+\|\tau\|_{H^2}^4+\|\nabla
u\|_{L^2}^4\right),
\end{eqnarray}
for some $\kappa_1>0$.

Next, taking the inner product of (\ref{Pu})
with $u$, we obtain
$$
\frac{\Rey}{2}\frac{d}{dt}\|u\|_{L^2}^2+(1-\alpha)\|\nabla u\|_{L^2}^2=(\dv \tau|u).
$$
Adding this equation to equation (\ref{tau_L^2}), integrating by parts and using the
fact that $\tau$ is symmetric, yields
\begin{eqnarray*}
\frac{1}{2}\frac{d}{dt}(\Rey\|u\|_{L^2}^2+\frac{\We}{2\alpha}\|\tau\|_{L^2}^2)
+ (1-\alpha)\|\nabla u\|_{L^2}^2+\frac{1}{2\alpha}\|\tau\|^2_{L^2}
&=&-\frac{\We}{2\alpha}(g_a(\tau, \nabla u)|\tau)\\
&\leq& \frac{C \We}{2\alpha}
\|\tau\|_{L^\infty}\|\nabla u\|_{L^2}\|\tau\|_{L^2}\\
&\leq&\frac{1-\alpha}{2}\|\nabla u\|_{L^2}^2+\frac{C \We^2}{(1-\alpha)\alpha^2}\|\tau\|^4_{H^2},
\end{eqnarray*}
which means that
\begin{equation}\label{utau_L^2}
\frac{d}{dt}(\Rey\|u\|_{L^2}^2+\frac{\We}{2\alpha}\|\tau\|_{L^2}^2) +
(1-\alpha)\|\nabla u\|_{L^2}^2+\frac{1}{\alpha}\|\tau\|^2_{L^2}
\leq\frac{C \We^2}{(1-\alpha)\alpha^2}\|\tau\|^4_{H^2}.
\end{equation}

Taking the inner product of (\ref{Pu}) with $Au$ yields
\begin{eqnarray*}
\frac{\Rey}{2}\frac{d}{dt}\|\nabla u\|_{L^2}^2 + (1-\alpha)\|Au\|_{L^2}^2
&\leq& \frac{1-\alpha}{2}\|Au\|^2_{L^2} + \frac{1}{1-\alpha}\|\mathbb{P}\dv\tau\|^2_{L^2} +
\frac{\Rey^2}{1-\alpha}\|\mathbb{P}(u\cdot\nabla )u\|^2_{L^2}.
\end{eqnarray*}
Further, since
\begin{equation}\label{convection}
\|\mathbb{P}(u\cdot\nabla )u\|^2_{L^2}\leq C\|(u\cdot\nabla )u\|^2_{L^2}\leq
C\|u\|_{L^6}^2\|\nabla u\|^2_{L^3}\leq C\|\nabla u\|_{L^2}^3\|\nabla u\|_{H^1},
\end{equation}
it follows that
\begin{equation}\label{nabla u_L^2}
\frac{d}{dt}(\Rey\|\nabla u\|_{L^2}^2) + (1-\alpha)\|Au\|_{L^2}^2
\leq\frac{2}{1-\alpha}\|\mathbb{P}\dv\tau\|^2_{L^2} + \varepsilon\|\nabla u\|_{H^1}^2 +
\frac{C \Rey^4}{\ve(1-\alpha)^2}\|\nabla u\|^6_{L^2}.
\end{equation}
Similarly, taking the inner product of (\ref{Pu}) with $u'$, and using (\ref{convection})
once more, we see that
\begin{equation}\label{nabla u_L^22}
\frac{d}{dt}((1-\alpha)\|\nabla u\|_{L^2}^2)+\Rey\|u'\|_{L^2}^2
\leq\frac{2}{\Rey}\|\mathbb{P}\dv\tau\|^2_{L^2}+\ve \|\nabla u\|_{H^1}^2 +
\frac{C \Rey^2}{\ve}\|\nabla u\|^6_{L^2}
\end{equation}
for $\ve>0$.
In view of (\ref{nabla^2uL^2}), (\ref{nabla u_L^2}) and (\ref{nabla u_L^22}) and by
choosing $\ve$ small enough, we are led to
\begin{eqnarray}\label{nabla u_L^222}
\nonumber&&\frac{d}{dt}\left((2\Rey+1-\alpha)\|\nabla
u\|_{L^2}^2\right) + \Rey\|\partial_tu\|_{L^2}^2 +
(1-\alpha)\|Au\|_{L^2}^2\\
&\leq&\kappa_2\left(\|\nabla
u\|_{L^2}^2+\|\mathbb{P}\dv\tau\|^2_{L^2}\right)+C\|\nabla
u\|^6_{L^2},
\end{eqnarray}
for some $\kappa_2>0$. Next, differentiating equations
$(\ref{OB})_1$ and $(\ref{OB})_3$ with respect to $t$, and taking
the inner product of the resulting equations with $\partial_t u$ and
$\partial_t\tau$, respectively, we obtain
\begin{equation*}
\frac{\Rey}{2}\frac{d}{dt}\|\partial_t u\|_{L^2}^2+(1-\alpha)\|\nabla
u_t\|_{L^2}^2=(\dv\tau_t|\partial_tu)-\Rey((u_t\cdot\nabla)u|\partial_tu),
\end{equation*}
as well as
\begin{eqnarray*}
\frac{\We}{2}\frac{d}{dt}\|\partial_t \tau\|_{L^2}^2+\|
\partial_t\tau\|_{L^2}^2&=&2\alpha(D(u_t)|\partial_t\tau) -
\We((u_t\cdot\nabla)\tau|\partial_t\tau)\\
&&-\We(g_a(\tau_t, \nabla u)|\tau_t)-\We(g_a(\tau,\nabla u_t)|\partial_t\tau).
\end{eqnarray*}
It follows that
\begin{eqnarray*}
&&\frac{1}{2}\frac{d}{dt}\left(\Rey\|\partial_tu\|_{L^2}^2+\frac{\We}{2\alpha}
\|\partial_t\tau\|_{L^2}^2\right)+(1-\alpha)\|\nabla u_t\|_{L^2}^2+\frac{1}{2\alpha}\|
\partial_t\tau\|_{L^2}^2\\
&\leq&C\|\nabla u\|_{H^2}(\Rey\|\partial_tu\|_{L^2}^2 +
\frac{\We}{2\alpha}\|\partial_t\tau\|_{L^2}^2)
+\frac{\We}{2\alpha}\|\partial_tu\|_{L^6}\|\nabla\tau\|_{L^3}\|\partial_t\tau\|_{L^2}\\
&&+C\frac{\We}{2\alpha}\|\nabla u_t\|_{L^2}\|\tau\|_{L^\infty}\|\partial_t\tau\|_{L^2}\\
&\leq&C\|\nabla u\|_{H^2}(\Rey\|\partial_tu\|_{L^2}^2 +
\frac{\We}{2\alpha}\|\partial_t\tau\|_{L^2}^2)
+ C \frac{\We}{2\alpha}\|\nabla u_t\|_{L^2}\|\tau\|_{H^2}\|\partial_t\tau\|_{L^2},
\end{eqnarray*}
and by Young's inequality that
\begin{eqnarray}\label{u_ttau_t}
\nonumber&&\frac{d}{dt}\left(\Rey\|\partial_tu\|_{L^2}^2 +
\frac{\We}{2\alpha}\|\partial_t\tau\|_{L^2}^2\right) +
(1-\alpha)\|\nabla u_t\|_{L^2}^2+\frac{1}{\alpha}\|\partial_t\tau\|_{L^2}^2\\
&\leq&\epsilon\|\nabla u\|_{H^2}^2+\frac{C}{\epsilon}(\Rey^2\|\partial_tu\|_{L^2}^4 +
\frac{\We^2}{4\alpha^2}\|\partial_t\tau\|_{L^2}^4)
+\frac{C \We^2}{\alpha^2(1-\alpha)}\|\tau\|_{H^2}^2\|\partial_t\tau\|_{L^2}^2.
\end{eqnarray}

Next, following an idea of Molinet and Talhouk \cite{MT04}, we estimate
$\hp\dv\tau$ and $\curl \dv\tau$, which will be then used in order to control
$\|\hp\dv\tau\|_{H^1}$. In order to  do so, we take the divergence of
$(\ref{OB})_3$ and,  using the incompressible condition, we obtain
$$
\We\left(\dv\tau_t+\dv((u\cdot\nabla)\tau)\right)+\dv\tau+\We\dv g_a(\tau, \nabla u)
=\alpha\Delta u,
$$
which, together with the equation of $u$, implies that
\begin{equation}\label{divtau}
\frac{1-\alpha}{\alpha}\We\dv\tau_t+\frac{1}{\alpha}\dv\tau=-\frac{1-\alpha}{\alpha}
\We\left[\dv((u\cdot\nabla)\tau)+\dv g_a(\tau,\nabla u)\right] +
\Rey\left[\partial_tu+(u\cdot\nabla)u\right]+\nabla p.
\end{equation}
Applying the Helmholtz projection  $\hp$ to this equation yields
\begin{equation*}
\frac{1-\alpha}{\alpha}\We\hp\dv\tau_t+\frac{1}{\alpha}\hp\dv\tau=
\Rey\partial_t u +
\hp\big[\Rey(u\cdot\nabla)u-\frac{1-\alpha}{\alpha}\We
\left(\dv((u\cdot\nabla)\tau)+\dv g_a(\tau,\nabla u)\right)\big].
\end{equation*}
Taking the inner product of the above equation with $\hp\dv\tau$ and
integrating by parts, we deduce that
\begin{eqnarray}
\nonumber&&\frac{1-\alpha}{\alpha}\frac{\We}{2}\frac{d}{dt}\|\hp\dv\tau\|_{L^2}^2 +
\frac{1}{\alpha}\|\hp\dv\tau\|_{L^2}^2 = -\Rey(\nabla u_t|\tau) +
\Rey(\hp(u\cdot\nabla)u|\hp\dv\tau)\\
&&\ \ \ -\frac{1-\alpha}{\alpha}\We(\hp\dv((u\cdot\nabla)\tau)|\hp\dv\tau) -
\frac{1-\alpha}{\alpha}\We(\hp\dv g_a(\tau,\nabla u)|\hp\dv\tau).
\end{eqnarray}
Note that
\begin{eqnarray*}
\Rey\left|(\hp((u\cdot\nabla)u)|\hp\dv\tau)\right|&\leq&
\Rey\|u\|_{L^6}\|\nabla u\|_{L^3}\|\tau\|_{H^1}
\leq C \Rey\|\nabla u\|_{L^2}\|\nabla u\|_{H^1}\|\tau\|_{H^1} \\
&\leq&\frac{\ve}{6}\|\nabla u\|_{H^1}^2 + \frac{C \Rey^2}{\ve} \|\nabla u\|_{L^2}^2
\|\tau\|_{H^1}^2,
\end{eqnarray*}
and
\begin{eqnarray*}
\frac{1-\alpha}{\alpha}\We\left|(\hp\dv((u\cdot\nabla)\tau)|\hp\dv\tau)\right|
&\leq& C\frac{1-\alpha}{\alpha}\We\left(\||\nabla u||\nabla \tau|\|_{L^2}+
\||u||\nabla^2\tau|\|_{L^2}\right)\|\tau\|_{H^1}\\
&\leq&C\frac{1-\alpha}{\alpha}\We\left(\|\nabla u\|_{L^6}\|\nabla \tau\|_{L^3} +
\|u\|_{L^\infty}\|\nabla^2\tau\|_{L^2}\right)\|\tau\|_{H^1}\\
&\leq&C\frac{1-\alpha}{\alpha}\We\|\nabla u\|_{H^1}\|\tau\|_{H^2}^2 \leq
\frac{\ve}{6}\|\nabla u\|_{H^1}^2 + \frac{C \We^2(1-\alpha)^2}{\ve\alpha^2}\|\tau\|_{H^2}^4,
\end{eqnarray*}
where we have used the Gagliardo-Nirenberg inequality (\ref{GN}) again.

Moreover, a direct calculation yields
$$
\frac{1-\alpha}{\alpha}\We\left|(\hp\dv g_a(\tau,\nabla u)|\hp\dv\tau)\right|\leq
\frac{\ve}{6}\|\nabla u\|_{H^1}^2+\frac{C \We^2(1-\alpha)^2}{\ve\alpha^2}\|\tau\|_{H^2}^4.
$$
which implies
\begin{equation}\label{divtau_L^2}
\frac{d}{dt}(\frac{1-\alpha}{\alpha}\We\|\hp\dv\tau\|_{L^2}^2) +
\frac{2}{\alpha}\|\hp\dv\tau\|_{L^2}^2
 \leq \Rey^2\|\nabla u_t\|_{L^2}^2+\|\tau\|_{L^2}^2+\ve\|\nabla u\|_{H^1}^2 +
C_\ve\left(\|\nabla u\|_{L^2}^4+\|\tau\|_{H^2}^4\right).
\end{equation}
Applying the $\curl$ operator to the equation (\ref{divtau}) we obtain
\begin{eqnarray*}
&&\frac{1-\alpha}{\alpha}\We \curl\dv\tau_t+\frac{1}{\alpha}\curl\dv\tau= \\
&& \qquad \Rey[\curl u_t+\curl((u\cdot\nabla)u)]-\frac{1-\alpha}{\alpha}
\We[\curl\dv((u\cdot\nabla)\tau)+\curl\dv g_a(\tau,\nabla u)].
\end{eqnarray*}
Taking the inner product of this  equation with $\curl\dv\tau$ yields
\begin{eqnarray}\label{curldiv}
&&\frac{1-\alpha}{\alpha}\frac{\We}{2}\frac{d}{dt}\|\curl\dv\tau\|_{L^2}^2 +
\frac{1}{\alpha}\|\curl\dv\tau\|_{L^2}^2 \leq \nn\\
&&\frac{1}{2\alpha}\|\curl\dv\tau\|_{L^2}^2
+\frac{C\alpha \Rey^2}{2}\|\nabla u_t\|_{L^2}^2
+\Rey(\curl((u\cdot\nabla)u)|\curl\dv\tau)- \nn\\
&&\frac{1-\alpha}{\alpha}\We(\curl\dv((u\cdot\nabla)\tau)|\curl\dv\tau)
- \frac{1-\alpha}{\alpha}\We(\curl\dv g_a(\tau, \nabla u)|\curl\dv\tau).
\end{eqnarray}
Noting  that
\begin{eqnarray*}
[\curl\dv((u\cdot\nabla)\tau)]_{lj}&=&(u\cdot\nabla)(\curl\dv\tau)_{lj}+
(\partial_{li}u^k\partial_k\tau^{ij}-\partial_{ji}u^k\partial_k\tau^{il})\\
&&+(\partial_iu^k\partial_{kl}\tau^{ij}-\partial_iu^k\partial_{kj}\tau^{il})+
(\partial_lu^k\partial_{ki}\tau^{ij}-\partial_ju^k\partial_{ki}\tau^{il}),
\end{eqnarray*}
we obtain
\begin{eqnarray*}
\frac{1-\alpha}{\alpha}\We\left|(\curl\dv((u\cdot\nabla)\tau)|\curl\dv\tau)\right|
&\leq& C\frac{1-\alpha}{\alpha}\We\left(\|\nabla^2u\|_{L^6}\|\nabla\tau\|_{L^3}+
\|\nabla u\|_{L^\infty}\|\nabla^2\tau\|_{L^2}\right)\|\nabla^2\tau\|_{L^2}\\
&\leq&\frac{\ve}{6}\|\nabla u\|_{H^2}^2+\frac{C(1-\alpha)^2\We^2}{\ve\alpha^2}\|\tau\|_{H^2}^4.
\end{eqnarray*}
Moreover,
\begin{eqnarray*}
\Rey\left|(\curl((u\cdot\nabla)u)|\curl\dv\tau)\right|&\leq&
C \Rey (\||\nabla u|^2\|_{L^2}+\||u||\nabla^2u|\|_{L^2})\|\nabla^2\tau\|_{L^2}\\
&\leq&C \Rey (\|\nabla u\|_{L^\infty}\|\nabla u\|_{L^2}+\|u\|_{L^6}\|\nabla^2u\|_{L^3})
\|\nabla^2\tau\|_{L^2}\\
&\leq&C \Rey \|\nabla u\|_{H^2}\|\nabla u\|_{L^2}\|\tau\|_{H^2}\leq\frac{\ve}{6}
\|\nabla u\|_{H^2}^2+\frac{C \Rey^2}{\ve}\|\nabla u\|_{L^2}^2\|\tau\|_{H^2}^2,
\end{eqnarray*}
and a direct computation yields
\begin{equation*}
\frac{1-\alpha}{\alpha}\We\left|(\curl\dv g_a(\tau, \nabla u)|\curl\dv\tau)\right|
\leq\frac{\ve}{6}\|\nabla u\|_{H^2}^2+\frac{C(1-\alpha)^2\We^2}{\ve\alpha^2}\|\tau\|_{H^2}^4.
\end{equation*}
Substituting these estimates in (\ref{curldiv}), we obtain
\begin{equation}\label{curldiv_L^2}
\frac{1-\alpha}{\alpha}\We\frac{d}{dt}\|\curl\dv\tau\|_{L^2}^2 +
\frac{1}{\alpha}\|\curl\dv\tau\|_{L^2}^2
\leq C\alpha \Rey^2\|\nabla u_t\|_{L^2}^2+\ve\|\nabla u\|_{H^2}^2 +
C_\ve(\|\nabla u\|_{L^2}^4+\|\tau\|_{H^2}^4).
\end{equation}
Putting together (\ref{divtau_L^2}) and (\ref{curldiv_L^2}) yields
\begin{eqnarray}\label{divtau curltau}
\frac{d}{dt}\left(\frac{1-\alpha}{\alpha}\We(\|\hp\dv\tau\|_{L^2}^2 +
\|\curl\dv\tau\|_{L^2}^2)\right) &+& \frac{1}{\alpha}
\left(\|\hp\dv\tau\|_{L^2}^2+\|\curl\dv\tau\|_{L^2}^2\right)\nn\\
&\leq&2\ve\|\nabla u\|_{H^2}^2+\kappa_3\|\nabla
u_t\|_{L^2}^2+\|\tau\|_{L^2}^2 +C_{\ve}(\|\nabla
u\|_{L^2}^4+\|\tau\|_{H^2}^4),
\end{eqnarray}
for some $\kappa_3>0$. On the other hand, multiplying equation
(\ref{nabla u_L^222}) with $\frac{\kappa_1+1}{1-\alpha}$ and adding
to (\ref{tau_H^22})
 yields
\begin{eqnarray}\label{taunablau}
\frac{d}{dt}(\We\|\tau\|_{H^2}^2 &+&\frac{(\kappa_1+1)(2\Rey+1-\alpha)}{1-\alpha}
\|\nabla u\|_{L^2}^2) +\frac{(\kappa_1+1)\Rey}{1-\alpha}\|\partial_tu\|_{L^2}^2+\|Au\|_{L^2}^2 +
\|\tau\|_{H^2}^2 + 2\|\nabla u\|_{H^2}^2 \nn\\
&\leq&\kappa_4\left(\|\nabla u\|_{L^2}^2+\|\nabla u_t\|_{L^2}^2+\|\hp\dv\tau\|_{H^1}^2\right)
+C\left(\|Au\|_{L^2}^4+\|\tau\|_{H^2}^4+\|\nabla
u\|_{L^2}^4+\|\nabla u\|_{L^2}^6\right),
\end{eqnarray}
for some $\kappa_4>0$.

Finally, we estimate $\|\hp\dv \tau\|_{H^1}$ in the right hand side
of (\ref{taunablau}). Notice first that in view of the Helmholtz
decomposition, we verify that $\curl\dv\tau=\curl\hp\dv\tau$.
Moreover, since $\dv(\hp\dv\tau)=0$ and $(\hp\dv\tau)\cdot\nu=0$, in
virtue of Proposition \ref{propunb}, there exists a constant $C_0$
such that
\begin{equation}\label{pdtau}
\|\hp\dv\tau\|_{H^1}^2\leq C_0\left(\|\hp\dv\tau\|_{L^2}^2+\|\curl\dv\tau\|_{L^2}^2\right).
\end{equation}
Then multiplying \eqref{divtau curltau} with $\alpha(\kappa_4C_0+1)$ and adding to
(\ref{taunablau}) implies that
\begin{eqnarray*}\label{taunablaupdtau}
&&\frac{d}{dt}\left(\We\|\tau\|_{H^2}^2+\frac{(\kappa_1+1)(2\Rey+1-\alpha)}{1-\alpha}
\|\nabla u\|_{L^2}^2+(1-\alpha)(\kappa_4C_0+1)\We(\|\hp\dv\tau\|_{L^2}^2
+\|\curl\dv\tau\|_{L^2}^2)\right)\nn \\
&&+\frac{(\kappa_1+1)\Rey}{1-\alpha}\|\partial_tu\|_{L^2}^2+\|Au\|_{L^2}^2
+\|\tau\|_{H^2}^2+2\|\nabla u\|_{H^2}^2+\|\hp\dv\tau\|_{L^2}^2+\|\curl\dv\tau\|_{L^2}^2\nn\\
&\leq&\kappa_5\left(\ve\|\nabla u\|_{H^2}^2+\|\nabla u\|_{L^2}^2+\|\nabla u_t\|_{L^2}^2
+\|\tau\|_{L^2}^2\right)+C_{\ve}\left(\|Au\|_{L^2}^4+\|\tau\|_{H^2}^4+\|\nabla
u\|_{L^2}^4+\|\nabla u\|_{L^2}^6\right),
\end{eqnarray*}
for some $\kappa_5>0$.

The term $\|\nabla u_t\|_{L^2}^2$ in the right hand side of above
can be absorbed into the left hand side by means of
(\ref{u_ttau_t}). Indeed, multiplying (\ref{u_ttau_t}) by
$\frac{\kappa_5+1}{1-\alpha}$, adding the resulting equation to
(\ref{taunablaupdtau}) and choosing
$\ve=\frac{1-\alpha}{\kappa_5(2-\alpha)+1}$, we infer that
\begin{eqnarray}\label{beforeend}
\nn&&\frac{d}{dt}\big(\We\|\tau\|_{H^2}^2+\frac{(\kappa_1+1)(2\Rey
+1-\alpha)}{1-\alpha}\|\nabla u\|_{L^2}^2+\frac{\Rey(\kappa_5+1)}{1-\alpha}
\|\partial_tu\|_{L^2}^2 + \frac{\We(\kappa_5+1)}{2\alpha(1-\alpha)}
\|\partial_t\tau\|_{L^2}^2 + \\
\nn&&(1-\alpha)(\kappa_4C_0+1)\We(\|\hp\dv\tau\|_{L^2}^2+\|\curl\dv\tau\|_{L^2}^2)\big)
+\frac{\Rey(\kappa_1+1)}{1-\alpha}\|\partial_tu\|_{L^2}^2 +\\
\nn&& \|Au\|_{L^2}^2+\|\tau\|_{H^2}^2+\|\nabla u\|_{H^2}^2+\|\nabla u_t\|_{L^2}^2
+\frac{(\kappa_5+1)}{\alpha(1-\alpha)}\|\partial_t\tau\|_{L^2}^2
+\|\hp\dv\tau\|_{L^2}^2+\|\curl\dv\tau\|_{L^2}^2\\
\nn&\leq&C\left(\|Au\|_{L^2}^4+\|\tau\|_{H^2}^4+\|\nabla u\|_{L^2}^4
+\|\nabla u\|_{L^2}^6+\|\partial_tu\|_{L^2}^4+\|\partial_t\tau\|_{L^2}^4\right)
+ \kappa_6\left(\|\nabla u\|_{L^2}^2+\|\tau\|_{L^2}^2\right),
\end{eqnarray}
for some $\kappa_6>0$. Finally, using the assumption $0<\alpha<1$
and multiplying \eqref{utau_L^2} with $\frac{\kappa_6+1}{1-\alpha}$
and adding it to (\ref{beforeend}) yields
\begin{eqnarray}\label{end}
\nn && \frac{d}{dt}\big(\frac{\Rey(\kappa_6+1)}{1-\alpha}\|u\|_{L^2}^2
+ \frac{\We(\kappa_6+1)}{2\alpha(1-\alpha)}\|\tau\|_{L^2}^2+\We\|\tau\|_{H^2}^2
+ \frac{(\kappa_1+1)(2\Rey+1-\alpha)}{1-\alpha}\|\nabla u\|_{L^2}^2 \nn\\
\nn && + \frac{\Rey(\kappa_5+1)}{1-\alpha}\|\partial_t u\|_{L^2}^2
+\frac{\We(\kappa_5+1)}{2\alpha(1-\alpha)}\|\partial_t\tau\|_{L^2}^2
+(1-\alpha)(\kappa_4C_0+1)\We(\|\hp\dv\tau\|_{L^2}^2+ \nn \\
\nn && \|\curl\dv\tau\|_{L^2}^2)\big) +\frac{\Rey(\kappa_1+1)}{1-\alpha}\|\partial_tu\|_{L^2}^2
+\|Au\|_{L^2}^2+\|\tau\|_{H^2}^2+\|\nabla u\|_{H^2}^2+\|\nabla u_t\|_{L^2}^2\\
\nn &&+ \frac{(\kappa_5+1)}{\alpha(1-\alpha)}\|\partial_t\tau\|_{L^2}^2
+\|\nabla u\|_{L^2}^2+\|\tau\|^2_{L^2}+\|\hp\dv\tau\|_{L^2}^2+\|\curl\dv\tau\|_{L^2}^2\\
&&\leq C\left(\|Au\|_{L^2}^4+\|\tau\|_{H^2}^4+\|\nabla u\|_{L^2}^4
+\|\nabla u\|_{L^2}^6+\|\partial_tu\|_{L^2}^4+\|\partial_t\tau\|_{L^2}^4\right).
\end{eqnarray}
After all these calculations, we are glad to define the functions $F:[0,\infty)\to \R$,
$G:[0,\infty)\to \R$ and $H:[0,\infty)\to \R$ as
\begin{eqnarray*}\label{F}
F(t)&:=&(1-\alpha)(\kappa_4C_0+1)\We(\|\hp\dv\tau\|_{L^2}^2+\|\curl\dv\tau\|_{L^2}^2)\\
\nn&&+\frac{\Rey(\kappa_6+1)}{1-\alpha}\|u\|_{L^2}^2
+\frac{\We(\kappa_6+1)}{2\alpha(1-\alpha)}\|\tau\|_{L^2}^2+\We\|\tau\|_{H^2}^2\\
\nn&&+\frac{(\kappa_1+1)(2\Rey+1-\alpha)}{1-\alpha}\|\nabla u\|_{L^2}^2
+\frac{\Rey(\kappa_5+1)}{1-\alpha}\|\partial_tu\|_{L^2}^2\\
\nn&&+\frac{\We(\kappa_5+1)}{2\alpha(1-\alpha)}\|\partial_t\tau\|_{L^2}^2,
\end{eqnarray*}
and
\begin{eqnarray*}\label{G}
 G(t)&:=&\frac{\Rey(\kappa_1+1)}{1-\alpha}\|\partial_tu\|_{L^2}^2
+\|Au\|_{L^2}^2+\|\tau\|_{H^2}^2+\|\nabla u\|_{H^2}^2+\|\nabla u_t\|_{L^2}^2\\
\nn&&+\frac{(\kappa_5+1)}{\alpha(1-\alpha)}\|\partial_t\tau\|_{L^2}^2
+\|\nabla u\|_{L^2}^2+\|\tau\|^2_{L^2}+\|\hp\dv\tau\|_{L^2}^2+\|\curl\dv\tau\|_{L^2}^2,
\end{eqnarray*}
and
\begin{eqnarray*}\label{H}
H(t):=\|\partial_tu\|_{L^2}^2+\|Au\|_{L^2}^2+\|\tau\|_{H^2}^2
+\|\partial_t\tau\|_{L^2}^2+\|\nabla u\|_{L^2}^2+\|\nabla u\|_{L^2}^4.
\end{eqnarray*}
We now rewrite (\ref{end}) as an inequality of the form
\begin{equation}\label{FHG}
\frac{d}{dt}F(t)+G(t)\leq C H(t)G(t)
\end{equation}
and estimate $H(t)$ in terms of $F(t)$. To this end, observe that by
 (\ref{Pu}), (\ref{convection}) and (\ref{nabla^2uL^2}) we deduce that
\begin{equation}\label{Au_L^2}
\|Au\|_{L^2}^2\leq C\left(\|\nabla u\|_{L^2}^2+\|\partial_tu\|_{L^2}^2+\|\hp\dv\tau\|_{L^2}^2+\|\nabla u\|_{L^2}^6\right).
\end{equation}
Hence, there exists a constant  $M_1=M_1(\Rey, \We, \alpha)>0$ such that
\begin{equation}\label{HF}
H(t)\leq M_1(F(t)+F(t)^2+F(t)^3), \quad t\geq 0.
\end{equation}
Substituting (\ref{HF}) into (\ref{FHG}), we get
\begin{eqnarray}\label{FFG}
\frac{d}{dt}F(t)+\left(1-CM_1(F(t)+F(t)^2+F(t)^3)\right)G(t)\leq 0, \quad t\geq 0.
\end{eqnarray}
We are now finally  in the  position to estimates $F(t)$. In order to do so,
define $\delta_0> 0$ small enough, such that
$\delta_0+\delta_0^2+\delta_0^3<\frac{1}{2CM_1}$, where $C$ is the constant appearing
in (\ref{FFG}).

Assume that the differential inequality (\ref{FFG}) holds for all $t\geq0$ and  $F$ being
absolutely continuous and $G$ being nonnegative. Then
\begin{equation}\label{assertion}
F(t)<\delta_0\ \ \mathrm{for\ all}\ \ t\geq0\ \ \mathrm{provided}\ \
F(0)<\delta_0.
\end{equation}
Assume that this assertion were not true. Let $t_1>0$ be the first
time where $F(t)\geq\delta_0$. Then
\begin{equation}\label{F(t_1)}
F(t_1)=\delta_0\ \ \mathrm{and}\ \ F(t)<\delta_0\ \ \mathrm{for\ all}\ \ 0\leq t<t_1.
\end{equation}
Consequently, for all $0\leq t\leq t_1$,
$$
1-CM_1(F(t)+F(t)^2+F(t)^3)\geq1-CM_1(\delta_0+\delta_0^2+\delta_0^3)>\frac12.
$$
Assertion (\ref{FFG}) implies now   that
\begin{eqnarray}\label{FG}
\frac{d}{dt}F(t)+\frac12G(t)\leq 0\ \ \mathrm{for\ all}\ \ 0\leq t\leq t_1.
\end{eqnarray}
Integrating (\ref{FG}) from $0$ to $t_1$, we obtain
\begin{eqnarray}\label{pfFG}
F(t_1)+\frac12\int_0^{t_1}G(s)ds\leq F(0)<\delta_0,
\end{eqnarray}
which contradicts (\ref{F(t_1)}). Thus (\ref{assertion}) holds true.

Given this fact, the proof of Theorem \ref{global} can now be finished easily. In fact, let
$T^\star$ be the lifespan of the local solution $(u,p,\tau)$ given in Proposition \ref{local}.
Assuming that (\ref{small}) holds with $\ve_0$ to be determined below, we verify that
\begin{equation*}
F(0)\leq
C(\|u_0\|_{H^1}^2+\|\tau_0\|_{H^2}^2+\|u_t(0)\|_{L^2}^2 +
\|\tau_t(0)\|_{L^2}^2) \leq C(\ve_0^2+\ve_0^4).
\end{equation*}
Choose now $\ve_0$ such that  $C(\ve_0^2+\ve_0^4)<\delta_0$.  Then,
in virtue of \eqref{assertion} and the proof of \eqref{pfFG},
 we obtain
\begin{equation}\label{gestimates}
\sup_{0\leq t\leq T^\star}F(t)+\frac12\int_0^{T^\star}G(t)dt<\delta_0.
\end{equation}
In particular, it follows from the definition  of $F(t)$ and $G(t)$,
(\ref{gestimates}), (\ref{Au_L^2}) and (\ref{nabla^2uL^2}) that
\begin{eqnarray*}
\nn && \sup_{0\leq t\leq
T^\star}\left(\|u(t)\|_{D(A)}^2+\|u'(t)\|_{L^2}^2+\|\tau(t)\|_{H^2}^2 +
\|\tau'(t)\|_{L^2}^2\right)\nn\\
\nn && +\int_0^{T^\star}(\|\nabla u(t)\|_{H^2}^2+\|\nabla
u'(t)\|_{L^2}^2+\| \tau(t)\|_{H^2}^2+\|
\tau'(t)\|_{L^2}^2)dt\leq C.
\end{eqnarray*}
We thus deduce  that the local solution $(u,p,\tau)$ can be extended for all positive times.
This completes the proof of Theorem \ref{global}.
$\Box$

\vn
{\bf Acknowledgments} This work was carried
out while the first and the third authors are visiting the
Department of Mathematics at the  Technical University of Darmstadt. They
would express their gratitude to Prof. Matthias Hieber for his kind
hospitality and the Deutsche Forschungsgemeinschaft (DFG) for financial support.

\noindent
We would like  thank Paolo Galdi for stimulating discussion concerning Oldroyd-B fluids and the
third author also would like to thank Tobias Hansel for his sincere help.

\noindent
Daoyuan Fang and Ruizhao Zi were partially supported by NSFC 10931007 and ZNSFC Z6100217.

\newpage

\end{document}